\documentclass[10pt, oneside]{article}
\usepackage{amsfonts, amsmath, amssymb, amscd, latexsym}
\usepackage{geometry}                
\usepackage{graphicx}
\usepackage{amssymb}
\usepackage{epstopdf}
\usepackage{hyperref}
\usepackage[utf8]{inputenc}
\usepackage{tikz}
\usepackage{mathtools}
\usepackage{array}

\usepackage{calrsfs}
\usepackage[cal=boondoxo]{mathalfa}

\usepackage{amsthm}
\usepackage{multirow}
\usepackage{tabu}
\usepackage{diagbox}

\usepackage{comment}
\usepackage{breqn}
\usepackage{nth}

\newtheorem*{property1}{Property N}
 
\newtheorem*{property2}{Property AU}

\newtheorem*{property3}{Property M}

\newtheorem{thm}{Theorem}

\newtheorem{prop}[thm]{Proposition}

\theoremstyle{definition}

\theoremstyle{definition}

\newcommand{\K}{\mathbb{K}}
\newcommand{\eme}{\mathcal{m}}
\newcommand{\F}{\mathbb{F}}

\author{Francisco Franco Munoz}
\date{}							
\title{\bf The existence of non-thin subalgebras of $\K[x]/x^n$ and related numerical monoids}

\begin{document}

\newcommand{\Addresses}{{
  \bigskip
  \footnotesize

  Francisco Franco Munoz, \textsc{Department of Mathematics and Statistics, Smith College,\\
    Northampton, MA 01060}\par\nopagebreak
  \textit{E-mail address}: \texttt{fmunoz@smith.edu}
}}

\maketitle

\abstract{We discuss examples of subalgebras of $\K[x]/x^n$, count them in the case of finite fields when $n\leq 10$, and emphasize the connections with monoids and their invariants. We prove that the least integer such that there is a ``non-thin" subalgebra of $\K[x]/x^n$ is $n=14$.}

\section{Setting and conventions}

In this work, the basic objects of study are $\K$-subalgebras of $\K[x]/x^n$, where $\K$ is a field. Such an algebra $R$ is local as a ring; we denote by $\eme= \eme_R$ its maximal ideal, $E=E(R)$ its set of exponents (i.e. the set of valuations of its monic elements), and $d(E)$ the minimal number of generators of $E$ as partial-monoid. We drop the subscripts and references to $R$ when understood from the context. We continue using here the same notation as in \cite{Franco-min} and we'll refer to that paper for background and results. All subspaces and dimensions are over $\K$.

\section{Summary of minimal extensions}

The relevant results of \cite{Franco-min} are applied in this setting as follows: Let $\varphi: \K[x]/x^{n+1}\to \K[x]/x^{n}$ the natural map. We'll consider subalgebras $R\subset \K[x]/x^{n}$. The main question we study is what can be said about the set of subalgebras mapping isomorphically onto $R$ by $\varphi$.

\begin{thm} The necessary and sufficient condition for the existence of a subalgebra $S\subset \K[x]/x^{n+1}$ mapping isomorphically onto $R$ is that $Ker(\varphi)$ (a one-dimensional $\K$-vector space) not be contained in $\eme_{\varphi^{-1}(R)}^2$. If that's the case, the set of such algebras is naturally in correspondence with an affine space, isomorphic to (i.e. same dimension as) $\eme_{R}/\eme_{R}^2$.
\end{thm}

In particular when $\K=\F_q$, there are $q^{d(R)}$ such subalgebras, where $d(R)=\dim(\eme_R/\eme_R^2)$. The process described in \cite{Franco-min} allows one to reduce to those subalgebras of $\K[x]/x^n$ containing $x^{n-1}$. Furthermore, as it's explained there, provided one knows that for all subalgebras of $\K[x]/x^{n}$ the equality $\dim(\eme_R/\eme_R^2)= d(E)$ holds, then one can count exactly all the subalgebras of $\K[x]/x^{n+1}$. 
\vspace{2 mm}

We recall here Proposition 22 of \cite{Franco-min} which says that $\dim(\eme_R/\eme_R^2)\leq d(E)$ and that's the most one can assert. An algebra for which equality holds will be called {\bf{thin}}, otherwise it's called {\bf{non-thin}}. Here's an example of a non-thin subalgebra \cite{Franco-min}:
\vspace{2 mm}

The algebra generated by $\{1, a=x^6+x^9, b=x^7, c=x^8\}$ inside $\K[x]/x^{18}$ has the basis $\{1, x^6+x^9, x^7, x^8, x^{12}, x^{13}, x^{14}, x^{15}, x^{16}, x^{17}\}$, and $E=\{0, 6, 7, 8, 12, 13, 14, 15, 16, 17\}$ and the generators of $E$ are $\{6, 7, 8, 17\}$ so $17$ is a generator of $E$ while $x^{17} = ac-b^2$ which belongs to $\eme^2$. We have also $\dim(\eme_R/\eme_R^2)=3 < 4=d(E)$. 

\section{Examples and tables}

\subsection{Examples in low dimensions}

Fix a field $\K$. We can list all the sub-partial-monoids $E$ of $[0, n-1]=\{0, 1, \cdots, n-1\}$ for low values of $n$ and compute $e(E)$ directly. Recall the definition of the invariant $e$:

  \[ e_{n}(E) = \begin{cases}
              0 &\text{if~} n=1 \\
               e_{n-1}(E\setminus \{n-1\})     & \text{if~} n-1\in E, n > 1\\
               d(E)+e_{n-1}(E)     & \text{if~} n-1\notin E, n > 1\\
            \end{cases} \]

 In the tables, the generators denote algebra generators, i.e. all the monomials in the elements (including the empty monomial $1$). It's not a priori immediate if those are linearly independent modulo $\eme^2$, but in the following lists that's the case and at the end we'll prove this theorem. 

\subsubsection{$n=1$}

\begin{center}
$\begin{array}{ | c | c | c |  } 
\hline
E & e(E) & \text{Subalgebras} \\ 
\hline
\{0\} & 0 & \K[x]/x=\K \\
\hline
\end{array}$
\end{center}

\subsubsection{$n=2$}

\begin{center}
$\begin{array}{ | c | c | c |  } 
\hline
E & e(E) & \text{Subalgebras} \\ 
\hline
\{0, 1\} & 0 & \K[x]/x^2 \\
\hline
\{0\} & 0 & \K \\
\hline
\end{array}$
\end{center}

\subsubsection{$n=3$}

\begin{center}
$\begin{array}{ | c | c | c |  } 
\hline
E & e(E) & \text{Subalgebras} \\ 
\hline
\{0, 1, 2\} & 0 & \K[x]/x^3 \\
\hline
\{0, 2\} & 0 & \{x^2\} \\
\hline
\{0\} & 0 & \K \\
\hline
\end{array}$
\end{center}

\subsubsection{$n=4$}

\begin{center}
$\begin{array}{ | c | c | c |  } 
\hline
E & e(E) & \text{Subalgebras} \\ 
\hline
\{0, 1, 2, 3\} & 0 & \K[x]/x^4 \\
\hline
\{0, 2, 3\} & 0 & \{x^2, x^3\} \\
\hline
\{0, 3\} & 0 & \{x^3\} \\
\hline
\{0, 2\} & 1 & \{x^2+ax^3\} \mid \text{any } a\in \K \\
\hline
\{0\} & 0 & \K\\
\hline
\end{array}$
\end{center}

\subsubsection{$n=5$}

\begin{center}
$\begin{array}{ | c | c | c |  }
\hline
E & e(E) & \text{Subalgebras} \\ 
\hline
\{0, 1, 2, 3, 4\} & 0 & \K[x]/x^5 \\
\hline
\{0, 2, 3, 4\} & 0 & \{x^2, x^3\} \\
\hline
\{0, 3, 4\} & 0 & \{x^3, x^4\} \\
\hline
\{0, 2, 4\} & 1 & \{x^2+ax^3\} \mid \text{any } a\in \K \\
\hline
\{0, 4\} & 0 & \{x^4\} \\
\hline
\{0, 3\} & 1 & \{x^3+ax^4\} \mid \text{any } a\in \K \\
\hline
\{0\} & 0 & \K \\
\hline
\end{array}$
\end{center}

\subsubsection{$n=6$}

\begin{center}
$\begin{array}{ | c | c | c |  }
\hline
E & e(E) & \text{Subalgebras} \\ 
\hline
\{0, 1, 2, 3, 4, 5\} & 0 & \K[x]/x^6 \\
\hline
\{0, 2, 3, 4, 5\} & 0 & \{x^2, x^3\} \\
\hline
\{0, 3, 4, 5\} & 0 & \{x^3, x^4, x^5\} \\
\hline
\{0, 2, 4, 5\} & 1 & \{x^2+ax^3, x^5\} \mid \text{any } a\in \K \\
\hline
\{0, 4, 5\} & 0 & \{x^4, x^5\} \\
\hline
\{0, 3, 5\} & 1  &  \{x^3+ax^4, x^5\} \mid \text{any } a\in \K \\
\hline
\{0, 5\} & 0 & \{x^5\} \\
\hline 
\{0, 3, 4\} & 2 & \{x^3+ax^5, x^4+bx^5\}  \mid \text{any } a, b\in \K \\
\hline
\{0, 2, 4\} & 2 & \{x^2+ax^3+bx^5\}  \mid \text{any } a, b\in \K \\
\hline
\{0, 4\} & 1 & \{x^4+ax^5\} \mid \text{any } a\in \K \\
\hline
\{0, 3\} & 2 & \{x^3+ax^4+bx^5\} \mid \text{any } a, b\in \K \\
\hline
\{0\} & 0 & \K \\
\hline
\end{array}$
\end{center}

\subsubsection{$n=7$}

\begin{center}
$\begin{array}{ | c | c | c |  }
\hline
E & e(E) & \text{Subalgebras} \\ 
\hline
\{0, 1, 2, 3, 4, 5, 6\} & 0 & \K[x]/x^7 \\
\hline
\{0, 2, 3, 4, 5, 6\} & 0 & \{x^2, x^3\} \\
\hline
\{0, 3, 4, 5, 6\} & 0 & \{x^3, x^4, x^5\} \\
\hline
\{0, 2, 4, 5, 6\} & 1 & \{x^2+ax^3, x^5\} \mid \text{any } a\in \K \\
\hline
\{0, 4, 5, 6\} & 0 & \{x^4, x^5, x^6\} \\
\hline
\{0, 3, 5, 6\} & 1  &  \{x^3+ax^4, x^5\} \mid \text{any } a\in \K \\
\hline
\{0, 5, 6\} & 0 & \{x^5, x^6\} \\
\hline
\{0, 3, 4, 6\} & 2 & \{x^3+ax^5, x^4+bx^5\}  \mid \text{any } a, b\in \K \\
\hline
\{0, 2, 4, 6\} & 2 & \{x^2+ax^3+bx^5\}  \mid \text{any } a, b\in \K \\
\hline
\{0, 4, 6\} & 1 & \{x^4+ax^5, x^6\} \mid \text{any } a\in \K \\
\hline
\{0, 3, 6\} & 2 & \{x^3+ax^4+bx^5\} \mid \text{any } a, b\in \K \\
\hline
\{0, 6\} & 0 & \{x^6\} \\
\hline 
\{0, 4, 5\} & 2 & \{x^4+ax^6, x^5+bx^6\} \mid \text{any } a, b\in \K \\
\hline
\{0, 5\} & 1 & \{x^5+ax^6\} \mid \text{any } a\in \K\\
\hline
\{0, 4\} & 2 & \{x^4+ax^5+bx^6\} \mid \text{any } a, b\in \K \\
\hline
\{0\} & 0 & \K \\
\hline
\end{array}$
\end{center}

\subsubsection{$n=8$}

\begin{center}
$\begin{array}{ | c | c | c |  }
\hline
E & e(E) & \text{Subalgebras} \\ 
\hline
\{0, 1, 2, 3, 4, 5, 6, 7\} & 0 & \K[x]/x^8 \\
\hline
\{0, 2, 3, 4, 5, 6, 7\} & 0 & \{x^2, x^3\} \\
\hline
\{0, 3, 4, 5, 6, 7\} & 0 & \{x^3, x^4, x^5\} \\
\hline
\{0, 2, 4, 5, 6, 7\} & 1 & \{x^2+ax^3, x^5\} \mid \text{any } a\in \K \\
\hline
\{0, 4, 5, 6, 7\} & 0 & \{x^4, x^5, x^6, x^7\} \\
\hline
\{0, 3, 5, 6, 7\} & 1  &  \{x^3+ax^4, x^5, x^7\} \mid \text{any } a\in \K \\
\hline
\{0, 5, 6, 7\} & 0 & \{x^5, x^6, x^7\} \\
\hline
\{0, 3, 4, 6, 7\} & 2 & \{x^3+ax^5, x^4+bx^5\}  \mid \text{any } a, b\in \K \\
\hline
\{0, 2, 4, 6, 7\} & 2 & \{x^2+ax^3+bx^5, x^7\}  \mid \text{any } a, b\in \K \\
\hline
\{0, 4, 6, 7\} & 1 & \{x^4+ax^5, x^6, x^7\} \mid \text{any } a\in \K \\
\hline
\{0, 3, 6, 7\} & 2 & \{x^3+ax^4+bx^5, x^7\} \mid \text{any } a, b\in \K \\
\hline
\{0, 6, 7\} & 0 & \{x^6, x^7\} \\
\hline
\{0, 4, 5, 7\} & 2 & \{x^4+ax^6, x^5+bx^6, x^7\} \mid \text{any } a, b\in \K \\
\hline
\{0, 5, 7\} & 1 & \{x^5+ax^6, x^7\}  \mid \text{any } a\in \K \\
\hline
\{0, 4, 7\} & 2 & \{x^4+ax^5+bx^6, x^7\} \mid \text{any } a, b\in \K \\
\hline
\{0, 7\} & 0 & \{x^7\} \\
\hline 
\{0, 4, 5, 6\} & 3 & \{x^4+ax^7, x^5+bx^7, x^6+cx^7\} \mid \text{any } a, b, c\in \K \\
\hline
\{0, 3, 5, 6\} & 3  &  \{x^3+ax^4+bx^7, x^5+cx^7\}  \mid \text{any } a, b, c\in \K \\
\hline
\{0, 5, 6\} & 2 & \{x^5+ax^7, x^6+bx^7\} \mid \text{any } a, b\in \K\\
\hline
\{0, 2, 4, 6\} & 3 & \{x^2+ax^3+bx^5+cx^7\}  \mid \text{any } a, b, c\in \K \\
\hline
\{0, 4, 6\} & 3 & \{x^4+ax^5+bx^7, x^6+cx^7\} \mid \text{any } a, b, c\in \K \\
\hline
\{0, 3, 6\} & 3 & \{x^3+ax^4+bx^5+cx^7\} \mid \text{any } a, b, c\in \K \\
\hline
\{0, 6\} & 1 & \{x^6+ax^7\} \mid \text{any } a\in \K \\
\hline
\{0, 4, 5\} & 4 & \{x^4+ax^6+cx^7, x^5+bx^6+dx^7\} \mid \text{any } a, b, c, d\in \K \\
\hline
\{0, 5\} & 2 & \{x^5+ax^6+bx^7\} \mid \text{any } a, b\in \K  \\
\hline
\{0, 4\} & 3 & \{x^4+ax^5+bx^6+cx^7\} \mid \text{any } a, b, c\in \K \\
\hline
\{0\} & 0 & \K \\
\hline
\end{array}$
\end{center}


\subsubsection{$n=9$}

\begin{center}
$\begin{array}{ | c | c | c |  }
\hline
E & e(E) & \text{Subalgebras} \\ 
\hline
\{0, 1, 2, 3, 4, 5, 6, 7, 8\} & 0 & \K[x]/x^9 \\
\hline
\{0, 2, 3, 4, 5, 6, 7, 8\} & 0 & \{x^2, x^3\} \\
\hline
\{0, 3, 4, 5, 6, 7, 8\} & 0 & \{x^3, x^4, x^5\} \\
\hline
\{0, 2, 4, 5, 6, 7, 8\} & 1 & \{x^2+ax^3, x^5\} \mid \text{any } a\in \K \\
\hline
\{0, 4, 5, 6, 7, 8\} & 0 & \{x^4, x^5, x^6, x^7\} \\
\hline
\{0, 3, 5, 6, 7, 8\} & 1  &  \{x^3+ax^4, x^5, x^7\} \mid \text{any } a\in \K  \\
\hline
\{0, 5, 6, 7, 8\} & 0 & \{x^5, x^6, x^7, x^8\} \\
\hline
\{0, 3, 4, 6, 7, 8\} & 2 & \{x^3+ax^5, x^4+bx^5\}  \mid \text{any } a, b\in \K \\
\hline
\{0, 2, 4, 6, 7, 8\} & 2 & \{x^2+ax^3+bx^5, x^7\}  \mid \text{any } a, b\in \K \\
\hline
\{0, 4, 6, 7, 8\} & 1 & \{x^4+ax^5, x^6, x^7\} \mid \text{any } a\in \K \\
\hline
\{0, 3, 6, 7, 8\} & 2 & \{x^3+ax^4+bx^5, x^7, x^8\} \mid \text{any } a, b\in \K \\
\hline
\{0, 6, 7, 8\} & 0 & \{x^6, x^7, x^8\} \\
\hline
\{0, 4, 5, 7, 8\} & 2 & \{x^4+ax^6, x^5+bx^6, x^7\} \mid \text{any } a, b\in \K \\
\hline
\{0, 5, 7, 8\} & 1 & \{x^5+ax^6, x^7, x^8\}  \mid \text{any } a\in \K \\
\hline
\{0, 4, 7, 8\} & 2 & \{x^4+ax^5+bx^6, x^7\} \mid \text{any } a, b\in \K \\
\hline
\{0, 7, 8\} & 0 & \{x^7, x^8\} \\
\hline 
\{0, 4, 5, 6, 8\} & 3 & \{x^4+ax^7, x^5+bx^7, x^6+cx^7\}  \mid \text{any } a, b, c\in \K \\
\hline
\{0, 3, 5, 6, 8\} & 3  &  \{x^3+ax^4+bx^7, x^5+cx^7\}  \mid \text{any } a, b, c\in \K \\
\hline
\{0, 5, 6, 8\} & 2 & \{x^5+ax^7, x^6+bx^7, x^8\}  \mid \text{any } a, b\in \K\\
\hline
\{0, 2, 4, 6, 8\} & 3 & \{x^2+ax^3+bx^5+cx^7\}  \mid \text{any } a, b, c\in \K \\
\hline
\{0, 4, 6, 8\} & 3 & \{x^4+ax^5+bx^7, x^6+cx^7\} \mid \text{any } a, b, c\in \K \\
\hline
\{0, 3, 6, 8\} & 3 & \{x^3+ax^4+bx^5+cx^7, x^8\} \mid \text{any } a, b, c\in \K \\
\hline
\{0, 6, 8\} & 1 & \{x^6+ax^7, x^8\} \mid \text{any } a\in \K \\
\hline
\{0, 4, 5, 8\} & 4 & \{x^4+ax^6+cx^7, x^5+bx^6+dx^7\} \mid \text{any } a, b, c, d\in \K \\
\hline
\{0, 5, 8\} & 2 & \{x^5+ax^6+bx^7, x^8\} \mid \text{any } a, b\in \K  \\
\hline
\{0, 4, 8\} & 3 & \{x^4+ax^5+bx^6+cx^7\} \mid \text{any } a, b, c\in \K \\
\hline
\{0, 8\} & 0 & \{x^8\} \\
\hline 
\{0, 5, 6, 7\} & 3 & \{x^5+ax^8, x^6+bx^8, x^7+cx^8\} \mid \text{any } a, b, c\in \K  \\
\hline
\{0, 3, 6, 7\} & 4 & \{x^3+ax^4+bx^5+cx^8, x^7+dx^8\} \mid \text{any } a, b, c, d\in \K  \\
\hline
\{0, 6, 7\} & 2 & \{x^6+ax^8, x^7+bx^8\} \mid \text{any } a, b\in \K \\
\hline
\{0, 5, 7\} & 3 & \{x^5+ax^6+bx^8, x^7+cx^8\}  \mid \text{any } a, b, c\in \K \\
\hline
\{0, 7\} & 1 & \{x^7+ax^8\} \mid \text{any } a\in \K \\
\hline
\{0, 5, 6\} & 4 & \{x^5+ax^7+cx^8, x^6+bx^7+dx^8\}  \mid \text{any } a, b, c, d\in \K \\
\hline
\{0, 3, 6\} & 4 & \{x^3+ax^4+bx^5+cx^7+dx^8\} \mid \text{any } a, b, c, d\in \K \\
\hline
\{0, 6\} & 2 & \{x^6+ax^7+bx^8\}  \mid \text{any } a, b\in \K \\
\hline
\{0, 5\} & 3 & \{x^5+ax^6+bx^7+cx^8\} \mid \text{any } a, b, c\in \K  \\
\hline
\{0\} & 0 & \K \\
\hline
\end{array}$
\end{center}


\newpage

\subsubsection{$n=10$}

\begin{center}
$\begin{array}{ | c | c | c |  }
\hline
E & e(E) & \text{Subalgebras} \\ 
\hline
\{0, 1, 2, 3, 4, 5, 6, 7, 8, 9\} & 0 & \K[x]/x^{10} \\
\hline
\{0, 2, 3, 4, 5, 6, 7, 8, 9\} & 0 & \{x^2, x^3\} \\
\hline
\{0, 3, 4, 5, 6, 7, 8, 9\} & 0 & \{x^3, x^4, x^5\} \\
\hline
\{0, 2, 4, 5, 6, 7, 8, 9\} & 1 & \{x^2+ax^3, x^5\} \mid \text{any } a\in \K \\
\hline
\{0, 4, 5, 6, 7, 8, 9\} & 0 & \{x^4, x^5, x^6, x^7\} \\
\hline
\{0, 3, 5, 6, 7, 8, 9\} & 1  &  \{x^3+ax^4, x^5, x^7\} \mid \text{any } a\in \K  \\
\hline
\{0, 5, 6, 7, 8, 9\} & 0 & \{x^5, x^6, x^7, x^8, x^9\} \\
\hline
\{0, 3, 4, 6, 7, 8, 9\} & 2 & \{x^3+ax^5, x^4+bx^5, x^9\}  \mid \text{any } a, b\in \K \\
\hline
\{0, 2, 4, 6, 7, 8, 9\} & 2 & \{x^2+ax^3+bx^5, x^7\}  \mid \text{any } a, b\in \K \\
\hline
\{0, 4, 6, 7, 8, 9\} & 1 & \{x^4+ax^5, x^6, x^7, x^9\} \mid \text{any } a\in \K \\
\hline
\{0, 3, 6, 7, 8, 9\} & 2 & \{x^3+ax^4+bx^5, x^7, x^8\} \mid \text{any } a, b\in \K \\
\hline
\{0, 6, 7, 8, 9\} & 0 & \{x^6, x^7, x^8, x^9\} \\
\hline
\{0, 4, 5, 7, 8, 9\} & 2 & \{x^4+ax^6, x^5+bx^6, x^7\} \mid \text{any } a, b\in \K \\
\hline
\{0, 5, 7, 8, 9\} & 1 & \{x^5+ax^6, x^7, x^8, x^9\}  \mid \text{any } a\in \K \\
\hline
\{0, 4, 7, 8, 9\} & 2 & \{x^4+ax^5+bx^6, x^7, x^9\} \mid \text{any } a, b\in \K  \\
\hline
\{0, 7, 8, 9\} & 0 & \{x^7, x^8, x^9\} \\
\hline 
\{0, 4, 5, 6, 8, 9\} & 3 & \{x^4+ax^7, x^5+bx^7, x^6+cx^7\}  \mid \text{any } a, b, c\in \K \\
\hline
\{0, 3, 5, 6, 8, 9\} & 3  &  \{x^3+ax^4+bx^7, x^5+cx^7\}  \mid \text{any } a, b, c\in \K \\
\hline
\{0, 5, 6, 8, 9\} & 2 & \{x^5+ax^7, x^6+bx^7, x^8, x^9\}  \mid \text{any } a, b\in \K\\
\hline
\{0, 2, 4, 6, 8, 9\} & 3 & \{x^2+ax^3+bx^5+cx^7, x^9\}  \mid \text{any } a, b, c\in \K \\
\hline
\{0, 4, 6, 8, 9\} & 3 & \{x^4+ax^5+bx^7, x^6+cx^7, x^9\} \mid \text{any } a, b, c\in \K  \\
\hline
\{0, 3, 6, 8, 9\} & 3 & \{x^3+ax^4+bx^5+cx^7, x^8\} \mid \text{any } a, b, c\in \K \\
\hline
\{0, 6, 8, 9\} & 1 & \{x^6+ax^7, x^8, x^9\} \mid \text{any } a\in \K \\
\hline
\{0, 4, 5, 8, 9\} & 4 & \{x^4+ax^6+cx^7, x^5+bx^6+dx^7\} \mid \text{any } a, b, c, d\in \K \\
\hline
\{0, 5, 8, 9\} & 2 & \{x^5+ax^6+bx^7, x^8, x^9\} \mid \text{any } a, b\in \K  \\
\hline
\{0, 4, 8, 9\} & 3 & \{x^4+ax^5+bx^6+cx^7, x^9\} \mid \text{any } a, b, c\in \K \\
\hline
\{0, 8, 9\} & 0 & \{x^8, x^9\} \\
\hline 
\{0, 5, 6, 7, 9\} & 3 & \{x^5+ax^8, x^6+bx^8, x^7+cx^8, x^9\} \mid \text{any } a, b, c\in \K  \\
\hline
\{0, 3, 6, 7, 9\} & 4 & \{x^3+ax^4+bx^5+cx^8, x^7+dx^8\} \mid \text{any } a, b, c, d\in \K  \\
\hline
\{0, 6, 7, 9\} & 2 & \{x^6+ax^8, x^7+bx^8, x^9\} \mid \text{any } a, b\in \K \\
\hline
\{0, 5, 7, 9\} & 3 & \{x^5+ax^6+bx^8, x^7+cx^8, x^9\}  \mid \text{any } a, b, c\in \K \\
\hline
\{0, 7, 9\} & 1 & \{x^7+ax^8, x^9\} \mid \text{any } a\in \K \\
\hline
\{0, 5, 6, 9\} & 4 & \{x^5+ax^7+cx^8, x^6+bx^7+dx^8, x^9\}  \mid \text{any } a, b, c, d\in \K \\
\hline
\{0, 3, 6, 9\} & 4 & \{x^3+ax^4+bx^5+cx^7+dx^8\} \mid \text{any } a, b, c, d\in \K \\
\hline
\{0, 6, 9\} & 2 & \{x^6+ax^7+bx^8, x^9\} \mid \text{any } a, b\in \K \mid \text{any } a, b\in \K \\
\hline
\{0, 5, 9\} & 3 & \{x^5+ax^6+bx^7+cx^8, x^9\} \mid \text{any } a, b, c\in \K  \\
\hline
\{0, 9\} & 0 & \{x^9\} \\
\hline
\end{array}$
\end{center}


\begin{center}
$\begin{array}{ | c | c | c |  }
\hline
\{0, 5, 6, 7, 8\} & 4 & \{x^5+ax^9, x^6+bx^9, x^7+cx^9, x^8+dx^9\} \mid \text{any } a, b, c, d\in \K  \\
\hline
\{0, 4, 6, 7, 8\} & 4 & \{x^4+ax^5+bx^9, x^6+cx^9, x^7+dx^9\} \mid \text{any } a, b, c, d\in \K \\
\hline
\{0, 6, 7, 8\} & 3 & \{x^6+ax^9, x^7+bx^9, x^8+cx^9\} \mid \text{any } a, b, c\in \K  \\
\hline
\{0, 5, 7, 8\} & 4 & \{x^5+ax^6+bx^9, x^7+cx^9, x^8+dx^9\}  \mid \text{any } a, b, c, d\in \K \\
\hline
\{0, 4, 7, 8\} & 4 & \{x^4+ax^5+bx^6+cx^9, x^7+dx^9\} \mid \text{any } a, b, c, d\in \K \\
\hline
\{0, 7, 8\} & 2 & \{x^7+ax^9, x^8+bx^9\} \mid \text{any } a, b\in \K \\
\hline 
\{0, 5, 6, 8\} & 5 & \{x^5+ax^7+cx^9, x^6+bx^7+dx^9, x^8+ex^9\}  \mid \text{any } a, b, c, d, e\in \K\\
\hline
\{0, 2, 4, 6, 8\} & 4 & \{x^2+ax^3+bx^5+cx^7+dx^9\}  \mid \text{any } a, b, c, d\in \K \\
\hline
\{0, 4, 6, 8\} & 5 & \{x^4+ax^5+bx^7+dx^9, x^6+cx^7+ex^9\} \mid \text{any } a, b, c, d, e\in \K \\
\hline
\{0, 6, 8\} & 3 & \{x^6+ax^7+bx^9, x^8+cx^9\} \mid \text{any } a, b, c\in \K \\
\hline
\{0, 5, 8\} & 4 & \{x^5+ax^6+bx^7+cx^9, x^8+dx^9\} \mid \text{any } a, b, c, d\in \K  \\
\hline
\{0, 4, 8\} & 4 & \{x^4+ax^5+bx^6+cx^7+dx^9\} \mid \text{any } a, b, c, d\in \K \\
\hline
\{0, 8\} & 1 & \{x^8+ax^9\}  \mid \text{any } a\in \K \\
\hline 
\{0, 5, 6, 7\} & 6 & \{x^5+ax^8+dx^9, x^6+bx^8+ex^9, x^7+cx^8+fx^9\} \mid \text{any } a, b, c, d, e, f\in \K \\
\hline
\{0, 6, 7\} & 4 & \{x^6+ax^8+cx^9, x^7+bx^8+dx^9\} \mid \text{any } a, b, c, d\in \K \\
\hline
\{0, 5, 7\} & 5 & \{x^5+ax^6+bx^8+dx^9, x^7+cx^8+ex^9\}  \mid \text{any } a, b, c, d, e\in \K \\
\hline
\{0, 7\} & 2 & \{x^7+ax^8+bx^9\} \mid \text{any } a, b\in \K \\
\hline
\{0, 5, 6\} & 6 & \{x^5+ax^7+cx^8+ex^9, x^6+bx^7+dx^8+fx^9\}  \mid \text{any } a, b, c, d, e, f\in \K \\
\hline
\{0, 6\} & 3 & \{x^6+ax^7+bx^8+cx^9\} \mid \text{any } a, b, c\in \K \\
\hline
\{0, 5\} & 4 & \{x^5+ax^6+bx^7+cx^8+dx^9\} \mid \text{any } a, b, c, d\in \K  \\
\hline
\{0\} & 0 & \K \\
\hline
\end{array}$
\end{center}

\vspace{3 mm}

\begin{thm} The tables are correct up to $n=10$. Moreover, the equality $\dim(\eme/\eme^2) = d(E)$ holds, i.e. these algebras are all thin.
\end{thm}

\begin{proof} Checking the values for $e(E)$ is straightforward. The fact that the coefficients appear in the indicated way is a consequence of the theory of minimal extensions (\cite{Franco-min}), as mentioned before. For $E\subseteq [0, n-1]$, denote $E^{(2)}=(E_{> 0} + E_{> 0})\cap [1, n-1]$, those that are sum of two nonzero elements (it's possible that $E^{(2)}=\emptyset$). Notice that $d(E)=\#(E_{> 0}\setminus E^{(2)})$.

We need to prove the assertion that if $n-1\in E$ and $n-1\notin E^{(2)}$ implies that $x^{n-1}\notin \eme^2$. This in turn proves that the counts for the sub-partial-monoids not containing $n-1$ are correct, since they come from the results on minimal extensions from those containing $n-1$.

By immediate inspection the results for $n\leq 6$ hold. For $n=7$, by direct inspection of the table, for all $E\subseteq [0, n-1]$ that contain $n-1=6$, either $n-1\notin E^{(2)}$ or the square $\eme^2=0$, i.e. $E$ has trivial additive structure, so $E^{(2)}=\emptyset$ (which means that for $x, y \in E_{>0}$, $x+y\geq n$).

\begin{itemize}
\item $n=8$. The ones that are not immediately obvious are:

\begin{center}
$\begin{array}{ | c | c |  }
\hline
\{0, 3, 5, 6, 7\}   &  \{x^3+ax^4, x^5, x^7\} \mid \text{any } a\in \K \\
\hline
\{0, 2, 4, 6, 7\}  & \{x^2+ax^3+bx^5, x^7\}  \mid \text{any } a, b\in \K  \\
\hline
\{0, 3, 6, 7\}  & \{x^3+ax^4+bx^5, x^7\} \mid \text{any } a, b\in \K \\
\hline
\end{array}$
\end{center}

It's enough to show $x^7\notin \eme^2$ for the first two, since the third is contained in the first. For $\{x^3+ax^4, x^5, x^7\} $ the only nonzero product in $\eme^2$ is $(x^3+ax^4)^2$ which is not a multiple of $x^7$. For $ \{x^2+ax^3+bx^5, x^7\} $ a linear generating set of $\eme^2$ consists of the powers $w^2, w^3$ (where $w=x^2+ax^3+bx^5$) and clearly no such combination equals $x^7$.

\item $n=9$. The ones that are not immediately obvious are:

\begin{center}
$\begin{array}{ | c | c |  }
\hline
\{0, 3, 6, 7, 8\}  & \{x^3+ax^4+bx^5, x^7, x^8\} \mid \text{any } a, b\in \K \\
\hline
\{0, 3, 6, 8\}  & \{x^3+ax^4+bx^5+cx^7, x^8\} \mid \text{any } a, b, c\in \K \\
\hline
\end{array}$
\end{center}

It's enough to show $x^8\notin \eme^2$ for the first exponent set, since that contains the second. For $ \{x^3+ax^4+bx^5, x^7, x^8\} $ a generating set of $\eme^2$ consists of the element $w^2$ (where $w=x^3+ax^4+bx^5$) which is not a multiple of $x^8$.

\item $n=10$. The ones that are not so obvious are:

\begin{center}
$\begin{array}{ | c | c |  }
\hline
\{0, 4, 6, 7, 8, 9\} & \{x^4+ax^5, x^6, x^7, x^9\} \mid \text{any } a\in \K \\
\hline
\{0, 4, 7, 8, 9\}  & \{x^4+ax^5+bx^6, x^7, x^9\} \mid \text{any } a, b\in \K  \\
\hline
\{0, 2, 4, 6, 8, 9\} & \{x^2+ax^3+bx^5+cx^7, x^9\}  \mid \text{any } a, b, c\in \K  \\
\hline
\{0, 4, 6, 8, 9\} & \{x^4+ax^5+bx^7, x^6+cx^7, x^9\} \mid \text{any } a, b, c\in \K   \\
\hline
\end{array}$
\end{center}

It's enough to show $x^9\notin \eme^2$ for the first and third exponent set.  For $\{x^4+ax^5, x^6, x^7, x^9\}$, the only nonzero product in $\eme^2$ is $(x^4+ax^5)^2$ which is not a multiple of $x^9$. For  $\{x^2+ax^3+bx^5+cx^7, x^9\}$, a linear generating set of $\eme^2$ consists of the powers $w^2, w^3, w^4$ where $w=x^2+ax^3+bx^5+cx^7$ and such a combination is never equal to $x^9$, by looking at leading coefficients. 

\end{itemize} 

\end{proof}

\section{Counting Monoids and Algebras}

We can collect the information about the monoids in the following tables. They follow by simply counting the items in the tables before. The top horizontal row labels the possible values of $e=e(E)$ for $E$ a sub-partial-monoid of $[0, n-1]$, and the vertical left row labels the co-size $c$ of $E$, namely the size of its complement $\#([0, n-1]\setminus E)$. We start the tables with the trivial case $n=1$. The last column is the total count of a given co-size.

Many patterns can be explained from results in \cite{Franco-min} while others have to do with Frobenius numbers of numerical monoids \cite{RS}, in particular with relations between the genus (number of elements of the complement) and the Frobenius number (largest number not contained in the monoid).

\subsection{Monoid tables}
 
\[
\begin{array}{|c|c|c|}
\hline
c\backslash e & 0 & \text{Total co-size} \\
\cline{1-3}
0& 1 & 1\\
\cline{1-3}
\end{array}
\qquad
\begin{array}{|c|c|c|}
\hline
c\backslash e & 0 & \text{Total co-size} \\
\cline{1-3}
0& 1 & 1\\
\cline{1-3}
1 &  1 & 1\\
\cline{1-3}
\end{array}
\qquad
\begin{array}{|c|c|c|}
\hline
c\backslash e & 0 & \text{Total co-size} \\
\cline{1-3}
0& 1 & 1\\
\cline{1-3}
1 &  1 & 1\\
\cline{1-3}
2 & 1 & 1\\
\cline{1-3}
\end{array}
\qquad
\begin{array}{|c|c|c|c|}
\hline
c\backslash e & 0 & 1 & \text{Total co-size} \\
\cline{1-4}
0& 1 & - & 1\\
\cline{1-4}
1 &  1 & - & 1\\
\cline{1-4}
2 & 1 & 1 & 2\\
\cline{1-4}
3 & 1 & - & 1\\
\cline{1-4}
\end{array}
\]

\[
\begin{array}{|c|c|c|c|}
\hline
c\backslash e & 0 & 1 & \text{Total co-size} \\
\cline{1-4}
0& 1 & - & 1\\
\cline{1-4}
1 &  1 & - & 1\\
\cline{1-4}
2 & 1 & 1 & 2\\
\cline{1-4}
3 & 1 & 1 & 2\\
\cline{1-4}
4 & 1 & - & 1\\
\cline{1-4}
\end{array}
\qquad
\begin{array}{|c|c|c|c|c|}
\hline
c\backslash e & 0 & 1 & 2 & \text{Total co-size} \\
\cline{1-5}
0& 1 & - & - & 1\\
\cline{1-5}
1 &  1 & - & - & 1\\
\cline{1-5}
2 & 1 & 1 & - & 2\\
\cline{1-5}
3 & 1 & 1 & 2 & 4\\
\cline{1-5}
4 & 1 & 1 & 1 & 3\\
\cline{1-5}
5 & 1 & - & - & 1\\
\cline{1-5}
\end{array}
\qquad
\begin{array}{|c|c|c|c|c|}
\hline
c\backslash e & 0 & 1 & 2 & \text{Total co-size} \\
\cline{1-5}
0& 1 & - & - & 1\\
\cline{1-5}
1 &  1 & - & - & 1\\
\cline{1-5}
2 & 1 & 1 & - & 2\\
\cline{1-5}
3 & 1 & 1 & 2 & 4\\
\cline{1-5}
4 & 1 & 1 & 2 & 4\\
\cline{1-5}
5 & 1 & 1 & 1 & 3\\
\cline{1-5}
6 & 1 & - & - & 1\\
\cline{1-5}
\end{array}
\]

\[
\begin{array}{|c|c|c|c|c|c|c|}
\hline
c\backslash e & 0 & 1 & 2 & 3 & 4 & \text{Total co-size} \\
\cline{1-7}
0 & 1 & - & - & - & - & 1\\
\cline{1-7}
1 & 1 & - & - & - & - & 1\\
\cline{1-7}
2 & 1 & 1 & - & - & - & 2\\
\cline{1-7}
3 & 1 & 1 & 2 & - & - & 4\\
\cline{1-7}
4 & 1 & 1 & 2 & 3 & - & 7\\
\cline{1-7}
5 & 1 & 1 & 2 & 2 & 1 & 7\\
\cline{1-7}
6 & 1 & 1 & 1 & 1 & - & 4\\
\cline{1-7}
7 & 1 & - & - & - & - & 1\\
\cline{1-7}
\end{array}
\qquad
\begin{array}{|c|c|c|c|c|c|c|}
\hline
c\backslash e & 0 & 1 & 2 & 3 & 4 & \text{Total co-size} \\
\cline{1-7}
0 & 1 & - & - & - & - & 1\\
\cline{1-7}
1 & 1 & - & - & - & - & 1\\
\cline{1-7}
2 & 1 & 1 & - & - & - & 2\\
\cline{1-7}
3 & 1 & 1 & 2 & - & - & 4\\
\cline{1-7}
4 & 1 & 1 & 2 & 3 & - & 7\\
\cline{1-7}
5 & 1 & 1 & 2 & 3 & 2 & 9\\
\cline{1-7}
6 & 1 & 1 & 2 & 2 & 2 & 8\\
\cline{1-7}
7 & 1 & 1 & 1 & 1 & - & 4\\
\cline{1-7}
8 & 1 & - & - & - & - & 1\\
\cline{1-7}
\end{array}
\]

\[
\begin{array}{|c|c|c|c|c|c|c|c|c|}
\hline
c\backslash e & 0 & 1 & 2 & 3 & 4 & 5 & 6 & \text{Total co-size} \\
\cline{1-9}
0 & 1 & - & - & - & - & - & - & 1\\
\cline{1-9}
1 & 1 & - & - & - & - & - & - & 1\\
\cline{1-9}
2 & 1 & 1 & - & - & - & - & - & 2\\
\cline{1-9}
3 & 1 & 1 & 2 & - & - &- & - & 4\\
\cline{1-9}
4 & 1 & 1 & 2 & 3 & - & - & - & 7\\
\cline{1-9}
5 & 1 & 1 & 2 & 3 & 5 & - & - & 12\\
\cline{1-9}
6 & 1 & 1 & 2 & 3 & 4 & 2 & 1 & 14\\
\cline{1-9}
7 & 1 & 1 & 2 & 2 & 3 & 1 & 1 & 11\\
\cline{1-9}
8 & 1 & 1 & 1 & 1 & 1 & - & - & 5\\
\cline{1-9}
9 & 1 & - & - & - & - & - & - & 1\\
\cline{1-9}
\end{array}
\]

\subsection{Subalgebras of $\F_q[x]/x^n$}

From the tables before we get the following counts. We omit the codimension $c=0$ row as it's trivial and also omit the dimension $n=1$ column since $\K[x]/x^1=\K$.  


\vspace{1 mm}


  \noindent\begin{tabular}{|c|c|c|c|c|c|c|c|c|c|c|c|}
    \hline
      \diagbox{c=codim}{n=dim} & $2$& $3$& $4$& $5$& $6$  \\
    \cline{1-6}
    $1$& $1$& $1$& $1$& $1$& $1$\\
    \cline{1-6}
    $2$&  $-$& $1$& $q+1$& $q+1$& $q+1$\\
    \cline{1-6}
    $3$ & $-$& $-$& $1$& $q+1$& $2q^2+q+1$ \\
    \cline{1-6}
    $4$ &$-$ & $-$ & $-$ & $1$ & $q^2+q+1$\\
    \cline{1-6}
    $5$& $-$ & $-$& $-$& $-$ & $1$ \\
    \cline{1-6}
    $6$ & $-$ & $-$ & $-$ & $-$ & $-$\\
   \cline{1-6}
    $7$ & $-$ & $-$ & $-$ & $-$ & $-$ \\
    \cline{1-6}
    $8$ & $-$ & $-$ & $-$ & $-$ & $-$ \\
  \cline{1-6}
    $9$ & $-$ & $-$ & $-$ & $-$ & $-$  \\
   \cline{1-6}
  \end{tabular}
    
\vspace{2 mm} 

      \noindent\begin{tabular}{|c|c|c|c|c|}
    \hline
      \diagbox{c}{n} & $7$ & $8$ & $9$ & $10$  \\
    \cline{1-5}
    $1$&  $1$& $1$ & $1$ & $1$\\
    \cline{1-5}
    $2$& $q+1$ & $q+1$ & $q+1$ & $q+1$\\
    \cline{1-5}
    $3$ & $2q^2+q+1$ & $2q^2+q+1$ & $2q^2+q+1$ & $2q^2+q+1$ \\
    \cline{1-5}
    $4$ & $2q^2+q+1$ & $3q^3+2q^2+q+1$ & $3q^3+2q^2+q+1$ & $3q^3+2q^2+q+1$\\
    \cline{1-5}
    $5$&  $q^2+q+1$ & $q^4+2q^3+2q^2+q+1$ & $2q^4+3q^3+2q^2+q+1$ & $5q^4+3q^3+2q^2+q+1$\\
    \cline{1-5}
    $6$ & $1$ & $q^3+q^2+q+1$ & $2q^4+2q^3+2q^2+q+1$ & $q^6+2q^5+4q^4+3q^3+2q^2+q+1$\\
   \cline{1-5}
    $7$ & $-$ & $1$ & $q^3+q^2+q+1$ & $q^6+q^5+3q^4+2q^3+2q^2+q+1$\\
    \cline{1-5}
    $8$ & $-$ & $-$ & $1$ & $q^4+q^3+q^2+q+1$\\
  \cline{1-5}
    $9$ & $-$ & $-$ & $-$ & $1$ \\
   \cline{1-5}
  \end{tabular}
    
    \newpage
   
\section{Beyond $n=10$}   

For the analysis that follows, recall that we have defined $E^{(2)} = (E_{>0}+E_{>0})\cap [1, n-1]$. In Theorem 2 we have used and proved several properties that we'll now formalize.

\begin{property1} $E$ is a sub-partial-monoid of $[0, n-1]$ such that $n-1\in E$ and $n-1\notin E^{(2)}$. 
\end{property1}

We notice the following:

\begin{prop} Property N is equivalent to $E\setminus \{n-1\}$ is a sub-partial-monoid of $[0, n-2]$.
\end{prop}

By the theory of monoids, these correspond to submonoids of $\mathbb{N}$ with Frobenius number $f=n-1$. Since we need only to analyze the maximal ones, in turn these can be characterized by their intersection with $[1, h]$, where $\displaystyle h=\frac{f}{2} - 1$ if $f$ is even or $\displaystyle h = \frac{f-1}{2}$ if $f$ is odd. For more background on monoids and submonoids of $\mathbb{N}$ see \cite{RS}.

\vspace{1 mm}

We have seen the following being a key property of subalgebra $R$, its maximal ideal $\eme$ and its set $E$  (which is determined by $\eme$ clearly).

\begin{property3} The pair $(\eme, E)$ has property M if $n-1\in E$, and $n-1\notin E^{(2)}$ implies $x^{n-1}\notin \eme^2$.
\end{property3}

Suppose furthermore that we have given a linear basis of $\eme$ with elements $z_1$, ... , $z_{r-1}$, $z_{r} = x^{n-1}$. Then for $x^{n-1}$ to belong to $\eme^2$ and $n-1\notin E^{(2)}$, $z_r$ must be a linear combination of products $z_{i_j}z_{i_k}$ with $j, k < r$. Now, if $E$ were such that $E^{(2)}$ has no multiplicity as a set, then that can't happen, given that the valuation of a sum with each summand of different valuations is the minimum of those valuations. This proves:

\begin{prop} If $E^{(2)}$ has no multiplicity (i.e. $\{a,b\}, \{c, d\}$ in $E_{>0}$ such that $a+b=c+d < n$ implies $\{a,b\}=\{c,d\}$), then $E$ has Property M. 
\end{prop}

Furthermore notice that the empty set $E\cap [1, h] = \emptyset$ (where $h$ as before), corresponds to the submonoid of $\mathbb{N}$  given by $\{0, t, t+1, ...., f-1\}\cup [f+1, \infty)$, where $t=f/2+1$ if $f$ is even and $t=(f+1)/2$ otherwise. This is exactly the case of the ideal $m^2=0$, hence $x^{n-1}\notin \eme^2$. Thus we'll exclude the empty set from analysis. 
\vspace{2 mm}

A further analysis also leads to consider the following property: A pair $(\eme, E)$ consisting of an ideal $\eme$ and its set $E$ satisfies almost-uniqueness if the following holds:
\vspace{1 mm}

\begin{property2} There exists a linear basis of monic elements $z_e$ indexed by $e\in E$ such that for all pair of sets $\{a,b\}, \{c, d\}$ in $E_{>0}$ (of size $1$ or $2$) and such that $a+b=c+d < n$, taking the monic elements $z_a, z_b, z_c, z_d$ corresponding to them, the identity $z_az_b = z_cz_d$ holds.
\end{property2}

\begin{prop} AU implies M.
\end{prop}

Notice that a multiplicity free $E$ satisfies AU and both imply M.

\begin{proof} Same argument as above, follows by considering products with the same valuations, the main point is that there are no nontrivial cancellations.
\end{proof}


\subsection{$n=11$}

 By the remarks before, these partial-monoids correspond with ones having Frobenius number $f=10$. And those are in bijection with certain subsets of $[1,4]$, namely those partial-sub-monoids that when extended won't contain $10$. Within the set $\{1, 2, 3, 4\}=[1, 4]$, $E$ can't contain $1$ nor $2$ since they divide $10$, so only can contain $3, 4$ but can't contain both since then the extension would contain $3+3+4=10$. End up with $\{3\}$ or $\{4\}$ which produce the maximal ones below, which correspond to $\{3,6,8,9\}$ and $\{4,7,8,9\}$ (as subsets of $[1, 9]$).
\vspace{2 mm}

\begin{center}
$\begin{array}{ | c | c | c |  }
\hline 
\{0, 4, 7, 8, 9, 10\} & 2 & \{x^4+ax^5+bx^6, x^7, x^9, x^{10}\} \mid \text{any } a, b\in \K  \\
\hline
\{0, 3, 6, 8, 9, 10\} & 3 & \{x^3+ax^4+bx^5+cx^7, x^8, x^{10}\} \mid \text{any } a, b, c\in \K \\
\hline
\end{array}$
\end{center}

\vspace{1 mm}

For $\{x^4+ax^5+bx^6, x^7, x^9, x^{10}\}$, $\eme^2$ is spanned by $(x^4+ax^5+bx^6)^2$ hence $x^{10}\notin \eme^2$. 
\vspace{1 mm}

For $\{x^3+ax^4+bx^5+cx^7, x^8, x^{10}\}$, a linear generating set of $\eme^2$ consists of the powers $\{w^2, w^3\}$, where $w=x^3+ax^4+bx^5+cx^7$ and clearly such a combination is never equal to $x^{10}$.
\vspace{1 mm}

Alternatively, and more simply, notice that the sets $E^{(2)}$ involved don't have multiplicity. 


\subsection{$n=12$}

The possible sets for $n=12$, correspond with ones having Frobenius number $f=11$. Such is determined by a partial-sum closed subset of $\{1,2,3,4,5\} = [1,5]$ and its union with $11$ minus the complement in $[1,5]$.

Let's analyze the minimum positive integer of $E$ inside $[1, 5]$. Such a set can't contain $1$, if contains $2$ then contains $4$ and can't contain $3$ nor $5$ since $3+4+4=11$ and $5+4+2=11$, so if contains $2$, then it's $\{2,4\}$ that produces $\{0,2,4,6,8,10\}$. If doesn't contain 2, but contains $3$, then doesn't have $4$ nor $5$ since $3+4+4=11$ and $5+3+3=11$. Hence containing $3$ implies it's $\{0,3,6,7,9,10\}$. If contains $4$ might contain $5$, so two possible $\{0,4,5,8,9,10\}$ or $\{0,4,6,8,9,10\}$. Now if $E$ only contains $5$ from $[1,5]$ then have $\{0,5,7,8,9,10\}$. 

Produces the following possible sets for $n=12$:

\begin{center}
$\begin{array}{ | c | c | c |  }
\hline 
\{0,5,7,8,9,10,11\} &  1 & \{x^5+ax^6, x^7, x^8, x^9, x^{11}\}  \mid \text{any } a\in \K \\
\hline
\{0,4,6,8,9,10,11\} & 3 & \{x^4+ax^5+bx^7, x^6+cx^7, x^9, x^{11}\} \mid \text{any } a, b, c\in \K  \\
\hline
\{0,4,5,8,9,10,11\} & 4 & \{x^4+ax^6+cx^7, x^5+bx^6+dx^7, x^{11}\} \mid \text{any } a, b, c, d\in \K \\
\hline
\{0,3,6,7,9,10,11\} & 4 & \{x^3+ax^4+bx^5+cx^8, x^7+dx^8, x^{11}\} \mid \text{any } a, b, c, d\in \K  \\
\hline
\{0,2,4,6,8,10,11\} & 4 & \{x^2+ax^3+bx^5+cx^7+dx^9, x^{11}\}  \mid \text{any } a, b, c, d\in \K \\  
\hline
\end{array}$
\end{center}

Let's analyze one by one:

\begin{itemize}
\item $\{0,5,7,8,9,10,11\}$ has $E^{(2)} = \{5+5\}$, no multiplicity.
\item $\{0,4,6,8,9,10,11\}$ has $E^{(2)} = \{4+4, 4+6\}$, no multiplicity.
\item $\{0,4,5,8,9,10,11\}$ has $E^{(2)} = \{4+4, 4+5, 5+5\}$, no multiplicity.
\item $\{0,3,6,7,9,10,11\}$ has $E^{(2)} = \{3+3, 3+6, 3+7\}$, no multiplicity.
\item $\{0,2,4,6,8,10, 11\}$ and the algebra generators of $\eme$, $ \{x^2+ax^3+bx^5+cx^7+dx^9, x^{11}\}$. This satisfies property AU, given that's generated as algebra by two elements, namely $x^2+ax^3+bx^5+cx^7+dx^9$ and $x^{11}$, the latter which is a null element, namely, annihilates $\eme$. 
\end{itemize}


\subsection{$n=13$}

The possible sets for $n=13$, correspond with ones having Frobenius number $f=12$. Such is determined by a partial-sum closed subset of $\{1,2,3,4,5\}=[1, 5]$ and its union with $12$ minus the complement in $[1,5]$.

Let's analyze the minimum positive integer of $E$ inside $[1, 5]$. Need to avoid divisors of $12$, so none of $1,2,3,4$ work. Hence only $5$ which gives rise to $E=\{0, 5, 8, 9, 10, 11, 12\}$ and $E^{(2)} = \{5+5\}$ which has no multiplicity.


\subsection{$n=14$}

The possible sets for $n=14$, correspond with ones having Frobenius number $f=13$. Such is determined by a partial-sum closed subset of $\{1,2,3,4,5,6\}=[1, 6]$ and its union with $13$ minus the complement in $[1, 6]$. Let's analyze the minimum positive integer of $E$ inside $[1, 6]$.

\begin{itemize}
\item least is $2$, then contains $2,4,6$ can't contain any other element since $2+2+2+2+2+3=13$, $2+2+2+2+5=13$, hence it's $\{0,2,4,6,8,10,12\}$
\item least is $3$, then can't contain $4$ since $3+3+3+4=13$, can't contain $5$ since $3+5+5=13$, and so it's $\{0,3,6,8,9,11,12\}$
\item least is $4$ then can't contain $5$ since $4+4+5=13$ and so can either contain or not $6$, giving $\{0,4,7,8,10,11,12\}$ and $\{0,4,6,8,10,11,12\}$
\item least is $5$, can contain $6$, so $\{0,5,6,9,10,11,12\}$, $\{0,5,7,9,10,11,12\}$
\item least is $6$ gives $\{0,6,8,9,10,11,12\}$
\end{itemize}


{\bf{We exclude the analysis and algebra generators for $\{0,4,6,8,10,11,12,13\}$ from the following table in light of the next section (they turn out not to be always independent modulo $\eme^2$)}}. For the rest, the analysis shows the validity of property AU:

\begin{center}
$\begin{array}{ | c | c | c |  }
\hline
\{0,6,8,9,10,11,12,13\} & 1 & \{x^6+ax^7, x^8, x^9, x^{10}, x^{11}, x^{13}\} \mid \text{any } a\in \K \\
\hline
\{0,5,7,9,10,11,12,13\}& 5 & \{x^5+ax^6+bx^8+dx^9, x^7+cx^8+ex^9, x^9, x^{11}, x^{13}\}  \mid \text{any } a, b, c, d, e\in \K \\
\hline
\{0,5,6,9,10,11,12,13\}& 4 & \{x^5+ax^7+cx^8, x^6+bx^7+dx^8, x^9, x^{13}\}  \mid \text{any } a, b, c, d\in \K \\
\hline
\{0,4,7,8,10,11,12,13\}& 4 & \{x^4+ax^5+bx^6+cx^9, x^7+dx^9, x^{10}, x^{13}\} \mid \text{any } a, b, c, d\in \K \\
\hline
\{0,4,6,8,10,11,12,13\}& 4 & - \\
\hline
\{0,3,6,8,9,11,12,13\}& 5 & \{x^3+ax^4+bx^5+cx^7+dx^{10}, x^8+ex^{10}, x^{13}\} \mid \text{any } a, b, c, d, e\in \K  \\
\hline
\{0,2,4,6,8,10,12,13\}& 5 & \{x^2+ax^3+bx^5+cx^7+dx^9+ex^{11}, x^{13}\} \mid \text{any } a, b, c, d, e\in \K \\  
\hline
\end{array}$
\end{center}

Let's analyze one by one:

\begin{itemize}

\item $\{0,6,8,9,10,11,12,13\}$ has $E^{(2)} = \{6+6\}$, no multiplicity

\item $\{0,5,7,9,10,11,12,13\}$ has $E^{(2)}= \{5+5, 5+7\}$, no multiplicity

\item $\{0,5,6,9,10,11,12,13\}$ has $E^{(2)}=\{5+5, 5+6\}$, no multiplicity

\item $\{0,4,7,8,10,11,12,13\}$ has $E^{(2)} = \{4+4, 4+7, 4+8\}$, no multiplicity

\item $\{0,3,6,8,9,11,12,13\}$, here $\{x^3+ax^4+bx^5+cx^7+dx^{10}, x^8+ex^{10}, x^{13}\}$ are algebra generators of $\eme$. Here the (multi)set is $E^{(2)}=\{3+3, 3+6, 3+8, 3+9, 6+6\}$, and the multiples of $3$ are the only ones giving rise to non-unique sums, i.e. if $w=x^3+ax^4+bx^5+cx^7+dx^{10}$, then $w^3w^9=w^6w^6$. From here the property AU follows.

\item $\{0,2,4,6,8,10,12,13\}$, here $\{x^2+ax^3+bx^5+cx^7+dx^9+ex^{11}, x^{13}\}$ are algebra generators of $\eme$. This satisfies property AU, given the the algebra is generated by two elements, namely $x^2+ax^3+bx^5+cx^7+dx^9+ex^{11}$ and $x^{13}$ and the latter is a null element.

\end{itemize}


\section{Minimal dimension for the existence of a non-thin subalgebra is $n=14$}

\subsection{Main Result}

We have an explicit classification of the $\K$-subalgebras of $\K[x]/x^n$ for $n\leq 13$. The purpose of this section is to show that for $n=14$ we get a non-thin subalgebra.

\begin{thm} There exists a non-thin subalgebra of $\K[x]/x^{14}$, denoted $\cal{R}$. Hence $n=14$ is the minimal $n$ with this property.
\end{thm} 

\begin{proof} Since we have shown that all subalgebras are thin for $n\leq 13$, we need to construct one for $n=14$.

\vspace{1 mm}

Consider the partial-monoid $E= \{0,4,6,8,10,12,13\}$ and the subalgebra $\cal{R}$ generated by $\{a=x^4+x^5, b=x^6+x^7, c=x^{10}, d=x^{11}, e=x^{13} \}$ inside $\K[x]/x^{14}$, which has elements:
\begin{itemize}
\item $a=x^4+x^5$
\item $b=x^6+x^7$
\item $a^2=x^8+2x^9+x^{10}$
\item $ab=x^{10}+2x^{11}+x^{12}$
\item $a^3=x^{12}+3x^{13}$
\item $b^2=x^{12}+2x^{13}$
\end{itemize}

From here it's clear that the vector space with basis $\{1, a, b, a^2, ab, x^{12}, x^{13}\}$ is closed under multiplication and hence it's the sought after algebra, whose set of exponents is precisely $E$. Notice that $x^{13}\in \eme^2$ but $13\notin E^{(2)}$. We can write the formula in $\K[x]/x^{14}$ as $$(x^4+x^5)^3-(x^6+x^7)^2=x^{13}$$

Furthermore, the sizes are $\#(E_{>0})=6$, $\#(E^{(2)})=3$ so $d(E)=6-3=3$, and the dimensions are $\dim(\eme) = 6$, $\dim(\eme^2)=4$ and so $\dim(\eme/\eme^2)=2 < d(E)=3$. Also, the set $E$ does have multiplicity: $4+4+4=6+6$, which is to be expected given property AU.

\end{proof}


\subsection{Count of the number of thin subalgebras in dimension $14$}

We have a generating set as algebra, not necessarily they're independent modulo $\eme^2$ as we saw, $\{x^4+ax^5+bx^7+dx^9+fx^{11}, x^6+cx^7+ex^9+gx^{11}, x^{13}\}$. Notice $(x^4+ax^5+bx^7+dx^9+fx^{11})^3=(x^4+ax^5)^3=x^{12}+3ax^{13}$ and $(x^6+cx^7+ex^9+gx^{11})^2=(x^6+cx^7)^2=x^{12}+2cx^{13}$ and hence it's thin iff $3a=2c$ and $b, d, e, f, g$ arbitrary. Notice that since in any field, either $3$ or $2$ is not zero, we can always solve for either $a$ or $c$, thereby one of them is arbitrary and the other is determined. 

Similarly, for the set $\{0,4,6,8,10,11, 12,13\}$, a generating set (together with $1$) as an algebra is $\{x^4+ax^5+bx^7+dx^9, x^6+cx^7+ex^9, x^{11}, x^{13}\}$. And the identical calculation shows that the algebra is thin iff $3a=2c$ and $b, d, e$ are arbitrary. We obtain:

\begin{prop} With exponent set $\{0,4,6,8,10,12,13\}$, there are $q^7-q^6$ non-thin subalgebras of $\F_q[x]/x^{14}$ and $q^6$ thin ones. Also, there are $q^5-q^4$ non-thin subalgebras of $\F_q[x]/x^{14}$ with exponent set $\{0,4,6,8,10, 11, 12,13\}$ and $q^4$ thin ones.
\end{prop}

In light of the analysis for $n=14$ carried out previously, we have:

\begin{prop} For all partial-monoids $E$ except those containing $\{0,4,6,8,10,12,13\}$ (there are exactly two of those, namely $\{0,4,6,8,10,12,13\}$ and $\{0,4,6,8,10,11,12,13\}$), every subalgebra $R\subseteq \K[x]/x^{14}$ with exponent set $E=E(R)$ is thin.
\end{prop}

We have the table for algebras $R\subset \K[x]/x^{14}$ attached to (maximal) partial-monoid $\{0,4,6,8,10,11,12,13\}$:
\vspace{1 mm}

$\begin{array}{| c | c | c | c |  }
\hline
\dim(\eme/\eme^2) & \text{Generators}  & \text{type} \\ 
\hline
4 & \{x^4+ax^5+bx^7+dx^9, x^6+cx^7+ex^9, x^{11}, x^{13}\} \mid \text{any } a, b, c, d, e \text{ with } 3a=2a \in \K & \text{ thin} \\
\hline
3 & \{x^4+ax^5+bx^7+dx^9, x^6+cx^7+ex^9, x^{11}\} \mid \text{any } a, b, c, d, e \text{ with } 3a\neq 2a \in \K & \text{ non-thin}\\
\hline
\end{array}$

\subsection{Other results concerning non-thin subalgebras}

Notice that the algebra found $\cal{R}$ and its monoid $E$ has $d(E)=3$. And this is also minimal:

\begin{prop} If $d(E)=1$ or $d(E)=2$, then $\dim(\eme/\eme^2)=d(E)$. 
\end{prop}

\begin{proof} By the inequality $\dim(\eme_R/\eme_R^2)\leq d(E)$, the case $d(E)=1$ follows since $\eme/\eme^2$ is always nonzero. 

Assume for contradiction that $R$ is such that $\eme/\eme^2$ is one dimensional but $d(E)=2$. Then take monic elements $a, b$ with $\nu(a)$ the minimal valuation and $\nu(b)$ the other generator. By minimality of the valuation, since $\eme/\eme^2$ is one dimensional, $a$ generates $R$ as algebra hence $b$ must be a polynomial in $a$ (this is evident here, one reason being that $\eme$ is nilpotent, so the ideal generated by $a$ is the set of linear combinations of positive powers of $a$); and so the valuation $\nu(b)$ is a multiple of $\nu(a)$, which is a contradiction with $d(E)=2$.
\end{proof}

As a corollary to the proof, we get that for a non-thin subalgebra, the minimum dimension of $\eme/\eme^2$ must be $2$:

\begin{prop} If $\dim(\eme/\eme^2)=1$, then $\dim(\eme/\eme^2)=d(E)$.
\end{prop}

One final question to consider is the size of $E$, namely its number of elements. The set $E$ we found has $7$ elements. Here we'll prove that's the least one can do.

\begin{prop} If $\#(E)\leq 6$, then $\dim(\eme/\eme^2)=d(E)$. 
\end{prop}

\begin{proof} If there's a counterexample to the statement, with minimal size $n$, one would have $n-1\in E$, and $n-1\notin E^{(2)}$, i.e. $n-1$ is a generator. Furthermore, a counterexample with $n$ minimal satisfies $x^{n-1}\in \eme^2$ (Otherwise considering the algebra generated by the rest of the generators doesn't contain $x^{n-1}$ and projecting to $\K[x]/x^{n-1}$ gives an embedding and an example with $n-1 < n$.) Henceforth we assume this. By above, we can assume $d(E)\geq 3$. Denote the smallest element of $E$ by $a$. If $2a\geq n$, then all products of elements in $\eme$ are zero, i.e. $\eme^2=0$, and in that case, given the obvious inequality, $d(E)\leq \#(E)-1$, we have an equality $d(E)\leq \#(E)-1=\dim(\eme)=\dim(\eme/\eme^2)\leq d(E)$. Hence we can assume that's not the case, namely $d(E)\leq \#(E)-2$. And moreover, $n-1\neq 2a$ since $n-1$ is a generator of $E$, hence $E$ contains at least $4$ elements $\{0, a, 2a, n-1\}$, but if $\#(E)=4$, then this set would equal $E$ and $d(E)=2$, which is not the case. Also notice that the element $2a$ is always multiplicity-free, i.e. the equation $z+w=a+a\in E$ with $w, z \in E_{> 0}$ implies $z=w=a$.

Hence we arrive at $E$ has $5\leq \#(E)\leq 6$ elements and $d(E)\leq \#(E)-2$. In the two cases, we'll show that we end up with a multiplicity free $E^{(2)}$ which by Proposition 4, implies $(\eme, E)$ has property N, which will be the desired contradiction. Cases:
\begin{itemize} 

\item $\#(E)=5$. Then by assumption $3\leq d(E)$ and $d(E)\leq 5-2=3$, hence $d(E)=3$. Since $\#(E_{> 0})= 5-1=4$, this says that only one positive element is not a generator. By the above remark, that element is $a+a$, and so any other sum is larger than $n-1$. Hence $E^{(2)}=\{2a=a+a\}$ and $E$ has no multiplicity.

\item $\#(E)=6$. Now $3\leq d(E)\leq 6-2=4$. 

If $d(E)=4$, then as before $E^{(2)}=\{2a=a+a\}$ which has no multiplicity. 

If instead $d(E)=3$, then we have exactly two positive elements in $E$ that are not generators. Denote by $a_2$ the next smallest element of $E$. Notice that the smallest element in $E^{(2)}$ (as a set) is $a+a$ and the next smallest is $a+a_2$, furthermore, this second element is also multiplicity free. Indeed if $z+w=a+a_2$ with $w, z \in E_{> 0}$, and if both $z, w\geq a_2$, then $z+w\geq 2a_2 > a+a_2$ which is not the case. Hence one is $a$ and the other is $a_2$. We obtain that $E^{(2)}=\{a+a, a+a_2\}$, which has no multiplicity.
\end{itemize}

\end{proof}


As a further observation, with the same notation in the proof, notice that since $2a\in E$, then $a_2\leq 2a$ and so $a+a_2\leq 3a$, with equality iff $a_2=2a$. And so in the case that $a_2=b$ is a generator, $3a\notin E$ and $2b\notin E$. This shows how tight the set $E=\{0, 4, 6, 8, 10, 12, 13\}$ found is, since if $a=4, b=6$, then $3a=12$ and in fact $2b=12=3a$ does belong to $E$. We obtain that $\cal{R}$ is in a sense \emph{the} minimal non-thin subalgebra:

\begin{thm} The subalgebra $\cal{R}\subset \K[x]/x^{14}$ is a minimal non-thin subalgebra in the following ways: $n=14$ is minimal, $d(E)=3$ is minimal, $\#(E)=7$ is minimal and $\dim(\eme/\eme^2)=2$ is minimal.
\end{thm}

\Addresses

\end{document}